\documentclass{amsart}
\usepackage{amssymb,amsmath,amsthm}
\usepackage[utf8]{inputenc}
\usepackage[italian,english]{babel}
\usepackage{mathrsfs}
\usepackage[textmath,displaymath,floats,graphics,tightpage]{preview}
\usepackage[all]{xy}
\usepackage{pgf,tikz}
\usetikzlibrary{arrows}
\usetikzlibrary{positioning,automata}
\usepackage{float}
\usepackage[labelformat=empty]{caption,subfig}
\usepackage{MnSymbol}
\usepackage{bm}

\restylefloat{figure}

\theoremstyle{definition}
\newtheorem{defin}{Definition}[section]
\theoremstyle{plain}
\newtheorem{thm}[defin]{Theorem}

\newtheorem{prop}[defin]{Proposition}

\theoremstyle{remark}
\newtheorem{ex}[defin]{Example}
\newtheorem{remark}[defin]{Remark}

\newcounter{num}

\title{Homogeneous edge-colorings of graphs}

\author{Paola Bonacini}
\email{bonacini@dmi.unict.it}

\author{Maria Grazia Cinquegrani}
\email{cinquegrani@dmi.unict.it}

\author{Lucia Marino}
\email{lmarino@dmi.unict.it}

\address{Università degli Studi di Catania\\
  Viale A. Doria 6\
95125 Catania\\
Italy}

\begin{document}
\maketitle
\begin{abstract}
	Let $G=(V,E)$ be a multigraph without loops and for any $x\in V$ let $E(x)$ be the set of edges of $G$ incident to $x$. A \emph{homogeneous edge-coloring of $G$} is an
	  assignment of an integer  $m\ge 2$ and a coloring $c\colon E\rightarrow S$ of the edges of $G$ such that $|S|=m$ and for any $x\in V$, if $|E(x)|=mq_{x}+r_{x}$ with $0\le r_{x}<m$, there exists a partition of $E(x)$ in $r_{x}$ color classes of cardinality $q_{x}+1$  and other $m-r_{x}$ color classes of cardinality $q_{x}$. The \emph{homogeneous chromatic index} $\widetilde{\chi}(G)$ is the least $m$ for which there exists such a coloring. We determine $\widetilde{\chi}(G)$ in the case that  $G$ is  a complete multigraph, a tree or a complete bipartite multigraph.
\end{abstract}

\section{Introduction}

Let $G=(V,E)$ be a multigraph without loops (see \cite{B,Vo} as a reference). The usual definition of  coloring of the edges of a multigraph is a mapping from the set of edges $E$ into a finite set of colors such that two adjacent edges have different colors. The chromatic index $\chi'(G)$ is the minimum number of colors for which there exists such a coloring for $G$ (see \cite{V1,V2,V3}).

In \cite{GMV} Gionfriddo, Milazzo and Voloshin define a coloring of the edges of a multigraph by a mapping between $E$ and a set of colors  for which each non-pendant vertex, i.e. of degree at least $2$, is incident to at least two edges of the same color. In their paper they give the definition of upper chromatic index, $\overline{\chi}'(G)$, which is the maximum $k$ for which there exists an edge coloring with $k$ colors, and they determine $\overline{\chi}'(G)$ that for an arbitrary multigraph $G=(V,E)$. The study of such a coloring is related to the coloring theory of mixed hypergraphs (see \cite[Problem 13]{Vo2}).

In \cite{GAR} Gionfriddo, Amato and Ragusa proceed in the way of studying edge colorings of a multigraph in which each non-pendant vertex is incident to at least two edges of the same color. In particular, they give the definition of equipartite edge coloring of a multigraph $G$, where, fixed an integer $h$, they search for the maximum number of colors for which for any $x\in V$ there exists a partition of $E(x)$ in color classes of the same cardinality $h$ with the exception of one of smaller cardinality.

In this paper, proceeding in this direction, we give the definition of \emph{homogeneous edge-coloring} of a multigraph as an
assignment of an integer $m\ge 2$ and a coloring such that for any $x\in V$ $E(x)$ has a partition in $m$ classes of colors  whose cardinality differs for at most $1$, in the case that $|E(x)|\ge m$, and all the edges of $E(x)$ has different colors, in the case that $|E(x)|<m$. In particular, we search for the \emph{homogeneous chromatic index}, which is the minimum number of colors $\widetilde{\chi}(G)$ for which there exists such a  coloring and we prove that, if $G$ is  either a complete multigraph, a tree or a complete bipartite multigraph, then either $\widetilde{\chi}(G)=2$ or $\widetilde{\chi}(G)=3$.

\section*{Acknowledgement} 

We wish to express our deepest gratitude and appreciation to Professor Mario Gionfriddo, for introducing us to this problem, for his support, encouragement and for the many stimulating conversations and discussions.

\section{Homogeneous edge-coloring}

\begin{defin}
  \label{D:2}
  A \emph{homogeneous edge-coloring of $G$} (or
  $m$-homogeneous edge-coloring  of $G$) is an
  assignment of an integer
  $m\ge 2$ and a coloring $c\colon E\rightarrow S$ of the edges of $G$ such that $|S|=m$ and for any $x\in V$, if $|E(x)|=mq_{x}+r_{x}$ with $0\le r_{x}<m$, there exists a partition of $E(x)$ in $r_{x}$ color classes of cardinality $q_{x}+1$  and other $m-r_{x}$ color classes of cardinality $q_{x}$.
\end{defin}

\begin{remark} 
		If $|E(x)|<m$, then $q_{x}=0$ and $r_{x}=|E(x)|$: in this case, the previous definition implies that any two edges of $E(x)$ must be colored with different colors. 
\end{remark}

 Given $x\in V$ and $k\in \{1,\dots,m\}$ it may happen that $c(\sigma)\ne k$ for any $\sigma \in E(x)$. However, if $i,j\in c(E(x))$, with $i\ne j$, then the number of edges of $E(x)$ colored with $i$ and the number of edges of $E(x)$ colored with $j$ either are equal or differ by $1$.

\begin{defin}
  \label{D:3}
  Let $G=(V,E)$ be a graph. The \emph{homogeneous chromatic index}
  $\widetilde{\chi}(G)$ is the   minimum integer $m$  such that $G$ admits a
  $m$-homogeneous edge-coloring. 
\end{defin}

\begin{remark}
	It is useful to underline the following facts.
	\begin{enumerate}
		\item An edge-coloring of a graph $G$ is a homogeneous edge-coloring. In particular, $\widetilde{\chi}(G)\le \chi'(G)$, where $\chi'(G)$ is the chromatic index of $G$.
		\item A path $P_{n}$, with $n\ge 2$, is the graph with vertices $\{x_{1},\dots,x_{n}\}$ and edges $\{x_{i},x_{i+1}\}$ for $i=1,\dots,n-1$.
		It is easy to see that $\widetilde{\chi}(P_{n})=2=\chi'(P_{n})$.
		\item A cycle $C_{n}$, with $n\ge 3$, is the graph with vertices $\{x_{1},\dots,x_{n}\}$ and edges $\{x_{i},x_{i+1}\}$ for $i=1,\dots,n-1$ and $\{x_{n},x_{1}\}$. If $n$ is even, then $\widetilde{\chi}(C_{n})=2=\chi'(C_{n})$; if $n$ is odd, then $\widetilde{\chi}(C_{n})=3=\chi'(C_{n})$. 
		\item A star $S_{n}$ is the graph with vertices $\{x_{0},\dots,x_{n}\}$ and edges $\{x_{0},x_{i}\}$ for $i=1,\dots,n$. It is easy to see that $\widetilde{\chi}(S_{n})=2$. However, in this case $\widetilde{\chi}(S_{n})<\chi'(S_{n})=n$.
		\item A wheel $W_{n}$ is the graph with vertices $\{x_{1},\dots,x_{n}\}$, with $n\ge 4$, such that the subgraph induced by $\{x_{2},\dots,x_{n}\}$ is the cycle $C_{n-1}$ and $x_{1}$ is adjacent to the other vertices $\{x_{2},\dots,x_{n}\}$. If $S=\{1,2\}$, then the following mapping $c\colon E\rightarrow S$ is a $2$-homogeneous coloring of $W_{n}$:
		\begin{itemize}
			\item if $i\in \{2,\dots,n\}$, we define:
			\[
				c(\{x_1,x_i\})=
				\begin{cases}
					1 &\text{ if $i$ is even}\\
					2 &\text{ if $i$ is odd},				
				\end{cases}
			\]
			\item if $i\in \{1,\dots,n-1\}$, we define:
			\[
				c(\{x_i,x_{i+1}\})=
				\begin{cases}
					1 &\text{ if $i$ is odd}\\
					2 &\text{ if $i$ is even},				
				\end{cases}
			\]
			\item and
			\[
			c(\{x_{n},x_{2}\})=
			\begin{cases}
				1 &\text{ if $n$ is even}\\
				2 &\text{ if $n$ is odd}.				
			\end{cases}
			\]
		\end{itemize}
		This means that $\widetilde{\chi}(W_{n})=2$. In this case $\widetilde{\chi}(W_{n})<\chi'(W_{n})=n-1$ for $n\ge 4$.
	\end{enumerate}
\end{remark}

\begin{ex}
	As we have just seen the wheel $W_{5}$ admits this $2$-homogeneous edge-coloring: 
	
	
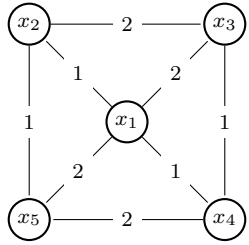
\begin{figure}[H]
	\begin{center}
		\begin{tikzpicture}[scale=0.8,font=\footnotesize,>=stealth',shorten >=1pt,node distance=1.3cm]
			\node[shape=circle,inner sep=2pt,draw,thick] (1) {$x_{1}$};
			\node[shape=circle,inner sep=2pt,draw,thick] (2) [left of=1, above of=1] {$x_{2}$};
			\node[shape=circle,inner sep=2pt,draw,thick] (3) [right of=1, above of=1]{$x_{3}$};
			\node[shape=circle,inner sep=2pt,draw,thick](4) [right of=1, below of=1] {$x_{4}$};
			\node[shape=circle,inner sep=2pt,draw,thick] (5) [left of=1, below of=1] {$x_{5}$};
			\path (1) edge node[fill=white]{$1$}(2)
					  edge node[fill=white]{$2$}(3)
					  edge node[fill=white]{$1$}(4)
					  edge node[fill=white]{$2$}(5);
			\path (2) edge node[fill=white]{$2$}(3)
					  edge node[fill=white]{$1$}(5);
			\path (3) edge node[fill=white]{$1$}(4);
			\path (4) edge node[fill=white]{$2$}(5);
		\end{tikzpicture}
	\end{center}
	\end{figure}
So $\widetilde{\chi}(W_{5})=2$, but $W_{5}$ has no $3$-homogeneous edge-coloring. In fact, the following are, up to permutation, the only possible colorings of $x_{1}$.

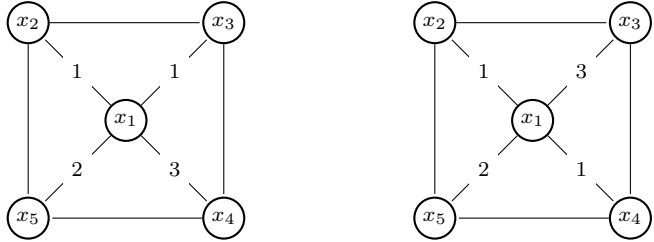
\begin{figure}[H]
\begin{center}
	\subfloat{
	\begin{tikzpicture}[scale=0.8,font=\footnotesize,>=stealth',shorten >=1pt,node distance=1.3cm]
		\node[shape=circle,inner sep=2pt,draw,thick] (1) {$x_{1}$};
		\node[shape=circle,inner sep=2pt,draw,thick] (2) [left of=1, above of=1] {$x_{2}$};
		\node[shape=circle,inner sep=2pt,draw,thick] (3) [right of=1, above of=1]{$x_{3}$};
		\node[shape=circle,inner sep=2pt,draw,thick](4) [right of=1, below of=1] {$x_{4}$};
		\node[shape=circle,inner sep=2pt,draw,thick] (5) [left of=1, below of=1] {$x_{5}$};
		\path (1) edge node[fill=white]{$1$}(2)
				  edge node[fill=white]{$1$}(3)
				  edge node[fill=white]{$3$}(4)
				  edge node[fill=white]{$2$}(5);
		\path (2) edge (3)
				  edge (5);
		\path (3) edge (4);
		\path (4) edge (5);
	\end{tikzpicture}}
\hspace{2cm}
\subfloat{
\begin{tikzpicture}[scale=0.8,font=\footnotesize,>=stealth',shorten >=1pt,node distance=1.3cm]
	\node[shape=circle,inner sep=2pt,draw,thick] (1) {$x_{1}$};
	\node[shape=circle,inner sep=2pt,draw,thick] (2) [left of=1, above of=1] {$x_{2}$};
	\node[shape=circle,inner sep=2pt,draw,thick] (3) [right of=1, above of=1]{$x_{3}$};
	\node[shape=circle,inner sep=2pt,draw,thick](4) [right of=1, below of=1] {$x_{4}$};
	\node[shape=circle,inner sep=2pt,draw,thick] (5) [left of=1, below of=1] {$x_{5}$};
	\path (1) edge node[fill=white]{$1$}(2)
			  edge node[fill=white]{$3$}(3)
			  edge node[fill=white]{$1$}(4)
			  edge node[fill=white]{$2$}(5);
	\path (2) edge (3)
			  edge (5);
	\path (3) edge (4);
	\path (4) edge (5);			
\end{tikzpicture}}
\end{center}
\end{figure}
The other vertices $x_{2}$, $x_{3}$, $x_{4}$, $x_{5}$ have degree $3$ and so in a $3$-homogeneous edge-coloring the three corresponding edges have three different colors. However this does not happen in any of the previous cases. 
\end{ex}

\begin{thm}
	If $G$ is an eulerian graph, then $G$ admits a $\tfrac{\Delta(G)}{2}$-homogeneous edge-coloring. 
\end{thm}
\begin{proof}
	This follows immediately by \cite[Theorem 2.1]{GAR}.
\end{proof}

\section{Complete graphs}

\begin{thm}
  \label{T:1}
  Let $n\ge 4$ be an even integer. Then $\widetilde{\chi}(K_n)=2$.
\end{thm}
\begin{proof}
Let $S=\{1,2\}$ be a set of colors and $n=2k$. Any vertex $x\in K_n$ has degree $n-1=2k-1$. So we will show
that, for any $x\in V$, we color $k$ edges of $E(x)$ with $1$ and
the remaining $k-1$ edges of $E(x)$ with $2$. Indeed, let
$V=\{x_1,\dots,x_n\}$ and let us consider a mapping $c\colon E\rightarrow S$ defined in the following way:
\[
	c(\{x_i,x_j\})=
	\begin{cases}
		1 & \text{ if $i+j$ is odd}\\
		2 & \text{ if $i+j$ is even,} 
	\end{cases}
\]
for any $i,j\in \{1,\dots,n\}$, with $i\ne j$. So, for any fixed $i$, we see that $k$ edges in $E(x_i)$ are colored with $1$ and $k-1$ edges of $E(x_j)$ are colored with $2$. This proves the statement.
\end{proof}

\begin{thm}
  \label{T:2}
  Let $n\ge 5$ be an odd integer such that $n\equiv 1\mod 4$. Then $\widetilde{\chi}(K_n)=2$.
\end{thm}
\begin{proof}
Let $S=\{1,2\}$ be a set of colors. Any vertex $x$ of the graph $K_n$ has even degree $n-1$. We want to show that we can color $\tfrac{n-1}{2}$ of these edges with $1$ and
$\tfrac{n-1}{2}$ of these edges with $2$. By \cite[Theorem 1.2]{AG} we
see that there exist $\tfrac{n-1}{2}$ cycles of length $n$ that
decompose $K_n$. Since $\tfrac{n-1}{2}$ is even, we can color all the
edges of $\tfrac{n-1}{4}$ cycles with $1$ and all the
edges of $\tfrac{n-1}{4}$ cycles with $2$. This proves the statement.

Let $n=4h+1$ for some $h\in \mathbb N$. Another possible coloring is the following mapping $c\colon E\rightarrow S$:
\[
	c(\{x_{i},x_{j}\})=
	\begin{cases}
		1 &\text{ if }j\equiv i-1,\dots,i-h,i+1,\dots,i+h\mod n\\
		2 &\text{ otherwise,}
	\end{cases}
\]
for any $i,j\in \{1,\dots,n\}$, with $i\ne j$.
\end{proof}

\begin{prop}
  \label{P:1}
  Let $n\in \mathbb N$ be an odd integer and let $S=\{1,2\}$ be a set of colors. Then in any coloring $c$ of the edges of a cycle $C_n$ with $S$ there are
  precisely an odd number of vertices whose adjacent edges have the
  same color.  
\end{prop}
\begin{proof}
  Let $n=2k+1$. The proof works by induction on $k$. If $k=1$, the statement is
  clear. Let $k\ge 2$ be an odd integer and suppose that the statement
  holds for $k-1$. First note that, since $n$ is
  odd, there exists at least a vertex $v$ in $C_n$  whose adjacent edges have the
  same color. Let $V=\{x_1,\dots,x_n\}$ be the set of vertices of
  $C_n$. We can suppose that $v=x_1$ and that  $1$ is the color of its
  adjacent edges. Consider the cycle $C_{n-2}$ of vertices $\{x_3,\dots,x_n\}$ obtained by $C_n$
  adding the edge $\{x_n,x_3\}$ and color this cycle with $1$. By
  hypothesis on $x_1$ we know that $\{x_1,x_n\}$ and $\{x_1,x_2\}$ are
  both colored with $1$. So, if $r$ is the number of vertices of
  $C_n$ whose adjacent edges have the
  same color and $s$ is the number of vertices of $C_{n-2}$
  whose adjacent edges have the same color, we see that either $r=s+2$
  or $r=s$. Indeed:
\begin{itemize}
	\item if $c(\{x_2,x_3\})=2$ and $c(\{x_3,x_4\})=2$, then $r=s+2$;
	\item if $c(\{x_2,x_3\})=2$ and $c(\{x_3,x_4\})=1$, then $r=s$;
	\item if $c(\{x_2,x_3\})=1$, then $r=s+2$.
\end{itemize}
By applying the inductive hypothesis on $C_{n-2}$ we see that $s$ is odd and so $r$ is odd too.
\end{proof}

\begin{thm}
  \label{T:3}
  Let $n\ge 3$ be an odd integer such that $n\equiv 3\mod 4$. Then $\widetilde{\chi}(K_n)=3$. 
\end{thm}
\begin{proof}
  Let $n=4k+3$ and suppose that $\widetilde{\chi}(K_n)=2$. Take $S=\{1,2\}$. Then
  each vertex has $2k+1$ edges colored with $1$ and $2k+1$ colored
  with $2$. By \cite[Theorem 1.2]{AG} we see that there exist $2k+1$ cycles of length $n$ that decompose $K_n$. Moreover by Proposition \ref{P:1} in any of these
cycles there are precisely an odd number of vertices whose adjacent edges have the
  same color. This, together with the fact that the cycles of length $n$
  decomposing $K_n$ are in odd number, gives a contradiction with the
  hypothesis that  $\widetilde{\chi}(K_n)=2$, because there will be a vertex with at
  least $2k+3$ edges with the same color.

	Now we show that $\widetilde{\chi}(K_n)=3$. Let $S=\{1,2,3\}$ be a set of colors.

	If $3\mid n-1$, then $K_{n}$ is decomposed by $\tfrac{n-1}{2}$ cycles of length $n$ and we can color all the edges in $\tfrac{n-1}{6}$ cycles with $1$, all the edges in $\tfrac{n-1}{6}$ cycles with $2$ and all the edges in $\tfrac{n-1}{6}$ cycles with $3$. Since all the vertices of $K_{n}$ have degree $n-1$, this shows that $\widetilde{\chi}(K_n)=3$ when $3\mid n-1$.

	In the case that $3\mid n-1$, i.e. $n=12h+7$ for some $h\in \mathbb N$, another possible coloring is the following mapping $c\colon E\rightarrow S$:
	\[
		c(\{x_{i},x_{j}\})=
		\begin{cases}
			1 &\text{ if }j\equiv i-1,\dots,i-2h-1,i+1,\dots,i+2h+1\mod n\\
			2 &\text{ if }j\equiv i-2h-2,\dots,i-4h-2,i+2h+2,\dots,i+4h+2\mod n\\
			3 &\text{ if }j\equiv i-4h-3,\dots,i-6h-3,i+4h+3,\dots,i+6h+3\mod n.
		\end{cases}
	\]
	for any $i,j\in \{1,\dots,n\}$, with $i\ne j$.

	Let $n\equiv 2 \mod 3$ and let $V=\{x_1,\dots,x_n\}$. We define a mapping $c\colon E\rightarrow S$ in the following way:
	\begin{itemize}
		\item if $i,j<n$
		\[
			c(\{x_i,x_j\})=
			\begin{cases}
				1 &\text{ if }i+j\equiv 2 \mod 3\\
				2 &\text{ if }i+j\equiv 1 \mod 3\\
				3 &\text{ if }i+j\equiv 0 \mod 3				
			\end{cases}
		\] 
		\item if $i=n$ 
		\[
			c(\{x_i,x_j\})=
			\begin{cases}
				1 &\text{ if }j\equiv 1 \mod 3\\
				2 &\text{ if }j\equiv 2 \mod 3\\
				3 &\text{ if }j\equiv 0 \mod 3				
			\end{cases}
		\] 
		\item if $j=n$ 
		\[
			c(\{x_i,x_j\})=
			\begin{cases}
				1 &\text{ if }i\equiv 1 \mod 3\\
				2 &\text{ if }i\equiv 2 \mod 3\\
				3 &\text{ if }i\equiv 0 \mod 3.				
			\end{cases}
		\]
	\end{itemize}
	An easy computation shows that this gives the statement in the case $n\equiv 2 \mod 3$, i.e. $n=12h+11$ for some $h\in \mathbb N$. So given any vertex $x\in V$ the edges in $E(x)$ can be divided in three subsets, one with $4h+4$ and two with $4h+3$ elements, in such a way that all the edges in the same subset are colored either with $1$ or $2$ or $3$.
	
	Let $n\equiv 0\mod 3$, so that $n=12h+3$ for some $h\in \mathbb N$. By \cite[Theorem 1.2]{AG} we see that there are $6h+1$ cycles of length $n$ that decompose $K_{n}$. To prove the statement it is sufficient to color all the edges in $2h$ cycles with $1$, in other $2h$ cycles with $2$ and other $2h$ cycles with $3$. The edges in the last cycle can be colored alternatively with $1$ and $2$ with the exception of one edge colored with $3$, accordingly to the sequence $1,2,1,2,\dots,1,2,3$. So given any vertex $x\in V$ the edges in $E(x)$ can be divided in three subsets, one with $4h$ and two with $4h+1$ elements, in such a way that all the edges in the same subset are colored either with $1$ or $2$ or $3$.
	
	In this case another possible coloring is the following mapping $c\colon E\rightarrow S$:
	\[
		c(\{x_{i},x_{j}\})=
		\begin{cases}
			1 &\text{ if }i+j\equiv 2\mod 3\\
			2 &\text{ if }i+j\equiv 1\mod 3\\
			3 &\text{ if }i+j\equiv 0\mod 3.
		\end{cases}
	\]
	for any $i,j\in \{1,\dots,n\}$, with $i\ne j$.
\end{proof}

\begin{thm}
	Let $\lambda,n\in \mathbb N$. Then:
	\[
		\widetilde{\chi}(\lambda K_{n})=
		\begin{cases}
			3 & \text{if $\lambda$ is odd and $n\equiv 3 \mod 4$}\\
			2 & \text{otherwise}.
		\end{cases}
	\]
\end{thm}
\begin{proof}
	We will prove the following:
	\[
	\widetilde{\chi}(\lambda K_{n})=
	\begin{cases}
		2 & \text{if $n$ is even}\\
		2 & \text{if $n\equiv 1 \mod 4$}\\
		2 & \text{if $\lambda$ is even and $n\equiv 3 \mod 4$}\\
		3 & \text{if $\lambda$ is odd and $n\equiv 3\mod 4$}.
	\end{cases}
	\]
	\textbf{First case: $\bm n$ even.} Let $S=\{1,2\}$ a set of colors. Then by Theorem \ref{T:1} each copy of $K_{n}$ has a $2$-homogeneous edge-coloring. We can use this coloring for $\lfloor \tfrac{\lambda}{2}\rfloor$ copies of $K_{n}$ and for the remaining $\lceil \tfrac{\lambda}{2}\rceil$ copies of $K_{n}$ the coloring obtained by permuting $1$ and $2$ in the previous one.
	
	\textbf{Second case: $\bm{n \equiv 1 \mod 4}$.} This follows immediately by Theorem \ref{T:2}: each copy of $K_{n}$ has a $2$-homogeneous edge-coloring. We can use this coloring for all the copies of $K_{n}$.
	
	\textbf{Third case: $\bm \lambda$ is even and $\bm{n\equiv 3 \mod 4}$.} Let $S=\{1,2\}$ a set of colors. We can color $\tfrac{\lambda}{2}$ copies of $K_{n}$ with $1$ and $\tfrac{\lambda}{2}$ copies of $K_{n}$ with $2$.
	
	\textbf{Fourth case: $\bm{\lambda}$ is odd and $\bm{n\equiv 3\mod 4}$.} Proceeding as in Theorem \ref{T:3} we see that $\widetilde{\chi}(\lambda K_{n})>2$. We need to prove that $\widetilde{\chi}(\lambda K_{n})=3$. We denote by $S=\{1,2,3\}$ a set of colors and by $c$ be the coloring given in Theorem \ref{T:3}.
	
	If $3\mid n-1$, then we can use $c$ for each copy of $K_{n}$. 
	
	Let either $n\equiv 2\mod 3$ or $n\equiv 0\mod 3$. We consider $c_{{132}}$ the coloring obtained by $c$ and by the permutation of the colors $(1\, 3\, 2)$ and $c_{123}$ the coloring obtained by $c$ and by  the permutation of the colors $(1\, 2\, 3)$. 
	\begin{itemize}
		\item If $\lambda=3m$, for some $m\in \mathbb N$, we color $m$ copies of $K_{n}$ with $c$, $m$ with $c_{132}$ and $m$ with $c_{123}$.
		\item If $\lambda=3m+1$, for some $m\in \mathbb N$, we color $m+1$ copies of $K_{n}$ with $c$, $m$ with $c_{132}$ and $m$ with $c_{123}$.
		\item If $\lambda=3m+2$, for some $m\in \mathbb N$, we color $m+1$ copies of $K_{n}$ with $c$, $m+1$ with $c_{132}$ and $m$ with $c_{123}$.
	\end{itemize}
	   This proves the statement.
\end{proof}

\section{Trees and complete bipartite graphs}

\begin{thm}
	If $G=(V,E)$ is a tree and $|V|\ge 3$, then $\widetilde{\chi}(G)=2$.
\end{thm}
\begin{proof}
	Let $|V|=n$. We proceed by induction on $n$. If $n=3$, then $G$ is an open path and so $\widetilde{\chi}(G)=2$. Now let the statement hold for a tree with $n-1$ vertices. Let $x\in V$ be a pendant vertex and let $G'=G-x$. Then $G'$ is a tree with $n-1$ vertices and by induction $\widetilde{\chi}(G')=2$. Considered a $2$-homogeneous edge-coloring of $G'$, it is easy to get a $2$-homogeneous edge-coloring of $G$, because $d(x)=1$. 
\end{proof}

\begin{thm} \label{T:4}
	Given $m,n\in \mathbb N$, $\widetilde{\chi}(K_{n,m})=2$. 
\end{thm}	
\begin{proof}
	Given $S=\{1,2\}$, it is sufficient to consider the following $c\colon E\rightarrow S$:
	\[
		c(\{x_i,x_j\})=
		\begin{cases}
			1 &\text{ if $i+j$ is even}\\
			2 &\text{ if $i+j$ is odd},				
		\end{cases}
	\]
	for any $i,j\in \{1,\dots,n\}$, with $i\ne j$.
\end{proof}

\begin{thm}
	Given $\lambda,m,n\in \mathbb N$, $\widetilde{\chi}(\lambda K_{n,m})=2$. 
\end{thm}
\begin{proof}
	Let $c$ be the coloring of $K_{m,n}$ given in Theorem \ref{T:4} and $c'$ the coloring obtained by $c$ permuting $1$ and $2$. Then the statement follows by considering the following coloring of $\lambda K_{m,n}$: we use the coloring $c$ for $\lceil \tfrac{\lambda}{2}\rceil$ copies of $K_{m,n}$ and $c'$ for the remaining $\lfloor \tfrac{\lambda}{2}\rfloor$ copies of $K_{m,n}$.
\end{proof}

\end{document}